\documentclass{amsart}
\usepackage[english]{babel}
\usepackage[latin1]{inputenc}
\usepackage[dvips,final]{graphics}
\usepackage{amsmath,amsfonts,amssymb,amsthm,amscd,array,stmaryrd,mathrsfs, mathdots, epigraph}
\usepackage{arydshln}
\usepackage[makeroom]{cancel}
\usepackage{pstricks}
 \usepackage[all]{xy}
 \usepackage{url}
\usepackage{multirow, blkarray}
\usepackage{booktabs}
\usepackage{textcomp}
 \usepackage[final]{epsfig}
 \usepackage{color}
\theoremstyle{plain}
\newtheorem{thm}{Theorem}
\newtheorem*{thmc}{Theorem}
\newtheorem{lem}{Lemma}[section]
\newtheorem{cor}[lem]{Corollary}
\newtheorem{prop}[lem]{Proposition}
\newtheorem{defn}[lem]{Definition}
\newtheorem{conj}[lem]{Conjecture}

\theoremstyle{definition}
\newtheorem{rem}[lem]{Remark}
\newtheorem{ex}[lem]{Example}

\let\ssection=\section
\renewcommand{\section}{\setcounter{equation}{0}\ssection}

\newcommand{\R}{\mathbb{R}}
\newcommand{\Z}{\mathbb{Z}}
\newcommand{\C}{\mathbb{C}}

\newcommand{\RP}{{\mathbb{RP}}}

\newcommand{\CP}{{\mathbb{CP}}}

\newcommand{\cM}{\mathcal{M}}

\newcommand{\id}{\textup{Id}}
\newcommand{\tr}{\textup{tr }}

\newcommand{\Id}{\mathrm{Id}}
\newcommand{\SL}{\mathrm{SL}}
\newcommand{\PSL}{\mathrm{PSL}}

\newcommand{\thup}{\textup{th}}

\def\g{\gamma}

\begin{document}

\title[Partitions of unity in $\SL(2,\mathbb Z)$]{Partitions of unity in $\SL(2,\mathbb Z)$,
negative continued fractions,
and dissections of polygons}

\author{Valentin Ovsienko}
\address{
Valentin Ovsienko,
Centre national de la recherche scientifique,
Laboratoire de Math\'ematiques 
U.F.R. Sciences Exactes et Naturelles 
Moulin de la Housse - BP 1039 
51687 Reims cedex 2,
France}
\email{valentin.ovsienko@univ-reims.fr}



\begin{abstract}
We characterize sequences of positive integers
$(a_1,a_2,\ldots,a_n)$
for which the $2\times2$ matrix $\left(
\begin{array}{cc}
a_n&-1\\[4pt]
1&0
\end{array}
\right)
\left(
\begin{array}{cc}
a_{n-1}&-1\\[4pt]
1&0
\end{array}
\right)
\cdots
\left(
\begin{array}{cc}
a_1&-1\\[4pt]
1&0
\end{array}
\right)
$
is either the identity matrix $\Id$, its negative $-\Id$, or square root of $-\Id$.
This extends a theorem of Conway and Coxeter that classifies such solutions
subject to a total positivity restriction.
\end{abstract}

\maketitle

\thispagestyle{empty}

\section{Introduction and main results}\label{IntSec}

Let $M_n(a_1,\ldots,a_n)\in\SL(2,\Z)$ be the matrix defined by the product
\begin{equation}
\label{SLEq}
M_n(a_1,\ldots,a_n):=
\left(
\begin{array}{cc}
a_n&-1\\[4pt]
1&0
\end{array}
\right)
\left(
\begin{array}{cc}
a_{n-1}&-1\\[4pt]
1&0
\end{array}
\right)
\cdots
\left(
\begin{array}{cc}
a_1&-1\\[4pt]
1&0
\end{array}
\right),
\end{equation}
where $(a_1,a_2,\ldots,a_n)$ are positive integers.
In terms of the generators of $\SL(2,\Z)$
$$
S=\left(
\begin{array}{cc}
0&-1\\[4pt]
1&0
\end{array}
\right),
\qquad
T=\left(
\begin{array}{cc}
1&1\\[4pt]
0&1
\end{array}
\right),
$$
the matrix~(\ref{SLEq}) reads:
$M_n(a_1,\ldots,a_n)=T^{a_n}S\,T^{a_{n-1}}S\cdots{}T^{a_1}S$.
Every matrix $A\in\SL(2,\Z)$ can be written in the form~(\ref{SLEq})
in many different ways.

The goal of this paper is to describe all solutions of the following three equations
$$
\begin{array}{rccl}
M_n(a_1,\ldots,a_n)&=&\Id,&\qquad\qquad\qquad\qquad(\hbox{Problem~I})\\[2pt]
M_n(a_1,\ldots,a_n)&=&-\Id,&\qquad\qquad\qquad\qquad(\hbox{Problem~II})\\[2pt]
M_n(a_1,\ldots,a_n)^2&=&-\Id.&\qquad\qquad\qquad\qquad(\hbox{Problem~III})
\end{array}
$$

Problem~II, with a certain total positivity restriction, 
was studied in~\cite{Cox,CoCo} under the name of ``frieze patterns''.
The theorem of Conway and Coxeter~\cite{CoCo} establishes a 
one-to-one correspondence between 
the solutions of Problem~II
such that $a_1+a_2+\cdots+a_n=3n-6$, and triangulations of $n$-gons.
This class of solutions will be called {\it totally positive}.
Coxeter implicitly formulated Problem~II in full generality,
when he considered frieze patterns with zero and negative entries; see~\cite{CR}.

The following observations are obvious.

(a) Cyclic invariance: if $(a_1,a_2\ldots,a_n)$ is a solution of one of the above problems,
then $(a_2,\ldots,a_n,a_1)$ is also a solution of the same problem.
It is thus often convenient to consider $n$-periodic infinite sequences
$(a_i)_{i\in\Z}$ with the cyclic order convention $a_{i+n}=a_i$.
Note however, that
although the property of being a solution of Problem~III
is cyclically invariant, in this case the matrix $M_n(a_1,\ldots,a_n)$ changes under
cyclic permutation of $(a_1,\ldots,a_n)$.

(b)
The ``doubling'' $(a_1,\ldots,a_n,a_1,\ldots,a_n)$
of a solution of Problem~II is a solution of Problem~I, 
and the ``doubling'' of a solution of Problem~III is a solution of Problem~II.

(c)
A particular feature of Problem~III (distinguishing it from Problems I and II)
 is that it is equivalent to a single equation
 $
 \tr{}M_n(a_1,\ldots,a_n)=0.
$
This Diophantine equation was considered in~\cite{CoOv}, 
where the totally positive solutions were classified.

\subsection{The main result}
We introduce the following combinatorial notion.

\begin{defn}
\label{QuiD}
(a)
We call a $3d$-{\it dissection} a partition of a convex $n$-gon
into sub-polygons by means of pairwise non-crossing diagonals, 
such that the number of vertices of every sub-polygon is a multiple of $3$.

(b)
The {\it quiddity} of a $3d$-dissection of an $n$-gon is the
(cyclically ordered) $n$-tuple of numbers
$(a_1,\ldots,a_n)$ such that $a_i$ is the number of sub-polygons adjacent to $i$-th vertex
of the $n$-gon.
\end{defn}

In other words, a $3d$-dissection splits an $n$-gon into
triangles, hexagons, nonagons, dodecagons, etc.
Classical triangulations are a very particular case of a $3d$-dissection.
The notion of quiddity
is similar to that of Conway and  Coxeter~\cite{CoCo}.

We will also consider {\it centrally symmetric} $3d$-dissection of $2n$-gons.
Quiddities of such dissections are $n$-periodic, i.e.,
are doubled $n$-tuples of positive integers:
$(a_1,\ldots,a_n,a_1,\ldots,a_n)$.
We call a {\it half-quiddity} any $n$-tuple of consecutive numbers
$(a_i,a_{i+1},\ldots,a_{i+n-1})$ in a $n$-periodic dissection of a $2n$-gon.

The following statement, proved in Section~\ref{CombSec}, is our main result.

\begin{thm}
\label{SecondMainThm}
(i) 
Every quiddity of a $3d$-dissection of an $n$-gon is a solution of Problem~I or Problem~II.
Conversely, every solution of Problem~I or~II is a quiddity of
a $3d$-dissection of an $n$-gon.

(ii) 
A half-quiddity of a centrally symmetric $3d$-dissection which is a solution of Problem~II
is a solution of Problem~III, and
every solution of Problem~III is a
half-quiddity of a centrally symmetric $3d$-dissection of a $2n$-gon.
\end{thm}

To distinguish between the solutions of Problems~I and~II
in Part (i) of the theorem, one needs to count the
total number of sub-polygons with even number of vertices
($6$-gons, $12$-gons,...) in the chosen $3d$-dissection.
If this number is {\it odd},  then the corresponding quiddity a solutions of Problem~I,
otherwise, it is a solutions of Problem~II.

In order to explain how to construct a solution
of Problems~I-III starting from $3d$-dissections we give here a simple example.

\begin{ex}
\label{FirstEx}
Consider the following dissection of a tetradecagon
($n=14$) into $4$ triangles and $2$ hexagons.
$$
\xymatrix @!0 @R=0.30cm @C=0.33cm
 {
&&&&&&3\ar@{-}[dddddddddddd]\ar@{-}[llld]\ar@{-}[rrrd]\ar@{-}[rrrddddddddddd]&
\\
&&&2\ar@{-}[lldd]\ar@{-}[rrrddddddddddd]&&&&&& 1\ar@{-}[rrdd]\\
\\
&1\ar@{-}[ldd]&&&&&&&&&&2\ar@{-}[rdd]\ar@{-}[rdddd]\\
\\
2\ar@{-}[dd]\ar@{-}[rdddd]&&&&&&&&&&&&1\ar@{-}[dd]\\
\\
1\ar@{-}[rdd]&&&&&&&&&&&&2\ar@{-}[ldd]\\
\\
&2\ar@{-}[rrdd]&&&&&&&&&&1\ar@{-}[lldd]\\
\\
&&&1\ar@{-}[rrrd]&&&&&& 2\ar@{-}[llld]\\
&&&&&&3&
}
$$
Label its vertices by the numbers of adjacent sub-polygons.
Reading these numbers (anti-clockwise) along the border of the tetradecagon, one obtains
a solution of Problem~II
$$
(a_1,\ldots,a_{14})=(3,2,1,2,1,2,1,3,2,1,3,1,2,1).
$$
Furthermore, every half-sequence, for instance
$(a_1,\ldots,a_{7})=(3,2,1,2,1,2,1)$ is a solution of Problem~III,
since the $3d$-dissection is centrally symmetric.
\end{ex}

To the best of our knowledge, $3d$-dissections have not been considered
in the literature.
Let us mention that, since the work of Conway and Coxeter,
triangulations of various geometric objects play important role
in the subject; see, e.g.,~\cite{BM,BPT}.
Higher angulations of $n$-gons have also been considered;
see~\cite{BHJ,HJ}.

\subsection{The surgery operations}
We also give an inductive procedure of construction of all the
solutions of Problems~I-III.
Consider the following two families of ``local surgery'' operations
on solutions of Problems~I-III.

\begin{enumerate}
\item[(a)]
The operations of the first type insert $1$ into the sequence $(a_1,a_2,\ldots,a_n)$, 
increasing the two neighboring entries by $1$:
\begin{equation}
\label{FirstSur}
(a_1,\ldots,a_i,a_{i+1},\ldots,a_n)
\mapsto
(a_1,\ldots,a_i+1,\,1,\,a_{i+1}+1,\ldots,a_n).
\end{equation}
Within the cyclic ordering of $a_i$, the operation is defined for all
$1\leq{}i\leq{}n$.
The operations~(\ref{FirstSur}) preserve the set of solutions of 
each of the above problems.

\item[(b)]
The operations of the second type break one entry, $a_i$, replacing it by
$a'_i,a''_i\in\Z_{>0}$ such that
$$
a'_i+a''_i=a_i+1,
$$ 
and insert two consecutive $1$'s between them:
\begin{equation}
\label{SecSur}
(a_1,\ldots,a_i,\ldots,a_n)
\mapsto
(a_1,\ldots,a'_i,\,1,\,\,1,\,a''_i,\ldots,a_n).
\end{equation}
The operations~(\ref{SecSur}) exchange the sets of solutions of Problems~I and~II, 
and preserve the set of solutions of Problem~III.
\end{enumerate}

The crucial difference between these two classes of operations is that
every operation~(\ref{FirstSur}) increases the number of sub-polygons
of a $3d$-dissection by $1$, while an operation~(\ref{SecSur}) keeps this number
unchanged.
Indeed, an operation~(\ref{FirstSur}) consists in a gluing an extra
``exterior'' triangle, while an operation~(\ref{SecSur})
selects one sub-polygon and increases the number of its
vertices by $3$.
For more details, see Section~\ref{CombSec}.

Note that
the operations~(\ref{FirstSur}) are very well known.
They were used by Conway and Coxeter~\cite{CoCo};
see also~\cite{CH,BR} and many other sources.
In particular, the totally positive solutions of Problem~II are precisely the solutions obtained by 
a sequence of operations~(\ref{FirstSur}); see Appendix.
The operations~(\ref{SecSur}) seem to be new.
They change the combinatorial nature of solutions
(from triangulations to $3d$-dissections),
they also have a geometric meaning in terms of the homotopy class
of a curve on the projective line; see Section~\ref{RotSec}.

For a given $n$, there are exactly $n$ different operations of type~(\ref{FirstSur}),
while the total number of different operations of type~(\ref{SecSur})
is equal to $a_1+\cdots+a_n$.
Every operation~(\ref{FirstSur}) transforms~$n$ into~$n+1$, 
while every operation~(\ref{SecSur}) transforms~$n$ into~$n+3$.

The following statement, proved in Section~\ref{AlgoSec},
is an ``algorithmic version'' of Theorem~\ref{SecondMainThm}.

\begin{thm}
\label{TheMainThm}
If $n=3$, then  Problem~II has a unique solution:
\begin{equation}
\label{ElSolTwo}
(a_1,a_2,a_3)=(1,1,1),
\end{equation}
and  every solution of Problem I (resp. II)
can be obtained from~(\ref{ElSolTwo}) 
by a sequence of the operations~(\ref{FirstSur}) and~(\ref{SecSur}),
such that the total number of operations~(\ref{SecSur}) is odd (resp. even).
Conversely, every sequence of operations~(\ref{FirstSur}) and~(\ref{SecSur}) 
 applied to~(\ref{ElSolTwo}) produces a solution of Problem I or II.
\end{thm}

Note that,
unlike Problems~I and~II, a solution of
Problem~III can be reduced, i.e., such that it cannot be simplified
by applying the inverse of the operations~(\ref{FirstSur}) and~(\ref{SecSur}).
The simplest examples of a reduced solutions are $(1,2),\,(2,1)$,
for~$n=2$ and $(1,1,2,1,1)$, for~$n=5$.
One has
 \begin{equation}
\label{RelSL}
M_5(1,1,2,1,1)=\left(
\begin{array}{cc}
0&1\\[4pt]
-1&0
\end{array}
\right).
\end{equation}

\subsection{Motivations}
Matrices~(\ref{SLEq}) are ubiquitous,
they appear in many problems of number theory,
algebra, dynamics, mathematical physics, etc.
The following topics are motivated our study,
these topics and their relationship with
Problems~I--III deserve further study.

\begin{enumerate}

\item[(a)]
Consider the linear equation
\begin{equation}
\label{DEqEq}
V_{i-1}-a_iV_{i}+V_{i+1}=0,
\end{equation}
with (known) coefficients $(a_i)_{i\in\Z}$ and (indeterminate) sequence
$(V_i)_{i\in\Z}$.
It is often called the discrete Sturm-Liouville, Hill, or Schr\"odinger equation.
There is a one-to-one correspondence between
solutions of Problem~I (resp. II) 
and equations~(\ref{DEqEq}) with positive integer
$n$-periodic coefficients $a_i$, such that
every solution $(V_i)_{i\in\Z}$ of the equation is periodic (resp. antiperiodic): 
$$
V_{i+n}=V_i  \quad\hbox{ (resp. $V_{i+n}=-V_i$)},
$$ 
for all~$i$.
Indeed, one has:
$$
\begin{bmatrix}
V_{n+1}\\[4pt]
V_{n}
\end{bmatrix}
=M_n(a_1,\ldots,a_{n})
\begin{bmatrix}
V_{1}\\[4pt]
V_{0}
\end{bmatrix}.
$$
In this language, the totally positive solutions of Conway and Coxeter
correspond to non-oscillating equations~(\ref{DEqEq});
see Section~\ref{RotSec}.
Note also that
the matrix $M_n(a_1,\ldots,a_n)$ is called the monodromy matrix
of the equation~(\ref{DEqEq}).
It plays an import an role in
the theory of integrable systems; see~\cite{Skl}.

\item[(b)]
The theory of negative continued fractions
$$
[a_1,a_2,\ldots,a_n]=
a_1 - \cfrac{1}{a_2 
          - \cfrac{1}{\ddots - \cfrac{1}{a_n} } } 
$$
is relevant for the subject of this paper,
although in this theory one usually considers $a_i\geq2$
and the matrix $M_n(a_1,\ldots,a_n)$ is hyperbolic.
Some ideas of the theory have found application to Farey sequences;
see~\cite{Zag,HS,SVS,HJ} and the Appendix.

\item[(c)]
The classical moduli space
$$
\cM_{0,n} := \{(v_1,...,v_n) \in\CP^1\,\vert\,v_i\not= v_{i+1}\}\,/\,\PSL(2,\C)
$$
 of configurations of $n$ points in $\CP^1$.
As a $(n-3)$-dimensional algebraic variety it can be described by:
$$
\cM_{0,n}\simeq
\left\{
(a_1,\ldots,a_{n})\in\C^n\,\vert\,
M_n(a_1,\ldots,a_{n})=-\id
\right\}.
$$
For instance, for~$n=5$ the moduli space of configurations of $5$ points
$(v_1,v_2,v_3,v_4,v_5)$ can be described by $5$ cross-ratios:
$$
a_i:=\frac{(v_{i+1}-v_{i+4})(v_{i+2}-v_{i+3})}{(v_{i+1}-v_{i+2})(v_{i+2}-v_{i+3})},
$$
that satisfy the equation $M_5(a_1,a_2,a_3,a_4,a_5)=-\Id$.
For more details; see~\cite{SVS1,SVRS,Sop}.
Theorem~\ref{SecondMainThm} provides a set of rational points of~$\cM_{0,n}$;
see Section~\ref{RotSec} for a construction of the element of~$\cM_{0,n}$
associated with a solution of Problem~I or~II.

\item[(d)]
Combinatorics of Coxeter's frieze patterns~\cite{Cox,CoCo}.
Although this is not the main subject of the paper,
we outline in Section~\ref{CoCoSec} the class of Coxeter's friezes
corresponding to arbitrary solutions of Problems~II and~III.
Note also that classical Farey sequences can be understood as
very particular cases of Coxeter friezes; see~\cite{Cox} (and also~\cite{SVS}).
In particular, the index of a Farey sequence defined in~\cite{HS},
is a totally positive solution of Problem~II.
Coxeter's friezes is an active area of research; see~\cite{Sop} 
and references therein.

\item[(e)]
Every element of $\SL(2,\Z)$ can be 
written in the form~(\ref{SLEq}) for some positive integers $(a_1,\ldots,a_n)$
which is an interesting characteristic.
 We conjecture in Section~\ref{GeneratSec} 
that there is a canonical way to associate a $3d$-dissection
to every element of the group $\PSL(2,\Z)$.

\end{enumerate}

\subsection{Enumeration}
We formulate the problem of enumeration
of solutions of Problems~I--III.
Counting the number of $3d$-dissections
of an $n$-gon can give the upper bound.
Note that the totally positive solutions of Problem~II
are enumerated by triangulations of $n$-gons,
so that the total number of solutions is given by the Catalan numbers.
This follows from the Conway and Coxeter theorem
and the fact that a triangulation is determined by its quiddity.
We refer to~\cite{FN} for a general theorem on enumeration of polygon dissections.
However, since a $3d$-dissection is not completely characterized by its quiddity
(cf. Section~\ref{NUSec}), there are more dissections than solutions.
For a first enumeration test
for the set of solutions of Problem~III,
see Section~\ref{ExIIISn}.

\section{Proof of Theorem~\ref{TheMainThm}}\label{AlgoSec}

In this section we prove Theorem~\ref{TheMainThm} and give some of its easy corollaries.

\subsection{Induction basis}
Let us first consider the simplest cases.

a) If $n=2$, then the matrix $M_2(a_1,a_2)$ is as follows:
$$
\left(
\begin{array}{cc}
a_2&-1\\[4pt]
1&0
\end{array}
\right)
\left(
\begin{array}{cc}
a_1&-1\\[4pt]
1&0
\end{array}
\right)=
\left(
\begin{array}{cc}
a_1a_2-1&-a_2\\[4pt]
a_1&-1
\end{array}
\right),
$$
with $a_1,a_2>0$.
Since this matrix cannot be $\pm\Id$,
Problems~I and~II have no solutions.

b) Consider the case $n=3$ and assume that the sequence $(a_1,a_2,a_3)$ contains
two consecutive $1$'s.
Set $(a_1,a_2,a_3)=(a,1,1)$.
The matrix $M_3(a_1,a_2,a_3)$ is then given by
$$
\left(
\begin{array}{cc}
a&-1\\[4pt]
1&0
\end{array}
\right)
\left(
\begin{array}{cc}
1&-1\\[4pt]
1&0
\end{array}
\right)
\left(
\begin{array}{cc}
1&-1\\[4pt]
1&0
\end{array}
\right)
=
\left(
\begin{array}{cc}
-1&1-a\\[4pt]
0&-1
\end{array}
\right).
$$
Hence Problem~II has one solution $(1,1,1)$, corresponding to $a=1$,
while Problems~I has no solutions.

\subsection{Surgery operations on matrices}\label{SurSec}

Let us analyze how the operations~(\ref{FirstSur}) and~(\ref{SecSur})
act on the matrix~(\ref{SLEq}).
This is just an elementary computation.

An operation~(\ref{FirstSur}) replaces the product of two elementary matrices
$$
\left(
\begin{array}{cc}
a_{i+1}&-1\\[4pt]
1&0
\end{array}
\right)
\left(
\begin{array}{cc}
a_{i}&-1\\[4pt]
1&0
\end{array}
\right)
$$
in the expression for $M_n(a_1,\ldots,a_n)$ by
\begin{equation}
\label{Casea}
\begin{array}{rcl}
\left(
\begin{array}{cc}
a_{i+1}+1&-1\\[4pt]
1&0
\end{array}
\right)
\left(
\begin{array}{cc}
1&-1\\[4pt]
1&0
\end{array}
\right)
\left(
\begin{array}{cc}
a_{i}+1&-1\\[4pt]
1&0
\end{array}
\right)&=&
\left(
\begin{array}{cc}
a_ia_{i+1}-1&-a_{i}\\[4pt]
a_{i+1}&-1
\end{array}
\right)\\[15pt]
&=&
\left(
\begin{array}{cc}
a_{i+1}&-1\\[4pt]
1&0
\end{array}
\right)
\left(
\begin{array}{cc}
a_{i}&-1\\[4pt]
1&0
\end{array}
\right).
\end{array}
\end{equation}
Therefore, an operation~(\ref{FirstSur}) does not change the matrix:
$$
M_{n+1}(a_1,\ldots,a_i+1,\,1\,,a_{i+1}+1,\ldots,a_n)=M_n(a_1,\ldots,a_n).
$$
It follows that the operations~(\ref{FirstSur})
preserve the sets of solutions of Problems~I-III.

Consider now an operation~(\ref{SecSur}).
Since
\begin{equation}
\label{Caseb}
\left(
\begin{array}{cc}
a''_i&-1\\[4pt]
1&0
\end{array}
\right)
\left(
\begin{array}{cc}
1&-1\\[4pt]
1&0
\end{array}
\right)
\left(
\begin{array}{cc}
1&-1\\[4pt]
1&0
\end{array}
\right)
\left(
\begin{array}{cc}
a'_i&-1\\[4pt]
1&0
\end{array}
\right)=
\left(
\begin{array}{cc}
1-a'_i-a''_i&1\\[4pt]
-1&0
\end{array}
\right)=
-\left(
\begin{array}{cc}
a_i&-1\\[4pt]
1&0
\end{array}
\right),
\end{equation}
the matrix $M_n(a_1,\ldots,a_n)$ changes its sign.
If the number of the operations~(\ref{SecSur}) is even, then
the sequence of operations also preserves the set of solutions of Problems~I and~II.

\subsection{Induction step}

We need the following lemma, which was essentially proved in~\cite{CoCo}
for Problem~II.

\begin{lem}
\label{OneLem}
Given a solution $(a_1,\ldots,a_n)$ of Problem~I, II, or~III, 
there exists at least one value of $1\leq{}i\leq{}n$, such that $a_i=1$.
\end{lem}

\begin{proof}
Assume that $a_i\geq2$ for all $i$, and consider the solution $(V_i)_{i\in\Z}$
of the equation~(\ref{DEqEq}) with initial conditions $(V_0,V_1)=(0,1)$.
Since $V_{i+1}=a_iV_i-V_{i-1}$, we see by induction that $V_{i+1}>V_i$
for all $i$.
Therefore, the solution $(V_i)_{i\in\Z}$ grows and cannot be periodic.

The matrix $M_n(a_1,\ldots,a_n)$ is the monodromy matrix
of~(\ref{DEqEq}).
More precisely, let $(V_i)_{i\in\Z}$ be a solution
of the equation~(\ref{DEqEq}).
Then for the vector $(V_{i+1},V_{i})^t$, we have
$$
\begin{bmatrix}
V_{i+1}\\[4pt]
V_{i}
\end{bmatrix}
=\left(
\begin{array}{cc}
a_i&-1\\[4pt]
1&0
\end{array}
\right)
\begin{bmatrix}
V_{i}\\[4pt]
V_{i-1}
\end{bmatrix},
\qquad\ldots,\qquad
\begin{bmatrix}
V_{i+n}\\[4pt]
V_{i+n-1}
\end{bmatrix}
=M_n(a_i,\ldots,a_{i+n})
\begin{bmatrix}
V_{i}\\[4pt]
V_{i-1}
\end{bmatrix}.
$$
Suppose first that $M_n(a_1,\ldots,a_n)=\Id$.
Then every solution of~(\ref{DEqEq}) must be periodic, which is a contradiction.

If now $M_n(a_1,\ldots,a_n)=-\Id$ or $M_n(a_1,\ldots,a_n)^2=-\Id$, then we can
use the doubling argument to conclude that every solution of~(\ref{DEqEq})
must be $2n$-periodic or $4n$-periodic, respectively.
\end{proof}

We are ready to prove that every solution of Problems~I and~II
can be obtained from the elementary solution~(\ref{ElSolTwo})
by a sequence of the operations~(\ref{FirstSur}) and~(\ref{SecSur}).

Given a solution $(a_1,\ldots,a_n)$,
by Lemma~\ref{OneLem} there exists at least one coefficient $a_i$ which is equal to~$1$.
There are then two possibilities:

(a) both $a_{i-1},a_{i+1}\geq2$; 

(b) there are two consecutive $1$'s, say $a_i=a_{i+1}=1$,
i.e., the chosen solution has the following ``fragment'': $(\ldots,a_{i-1},1,1,a_{i+2},\ldots)$.

In the case (a), consider the $(n-1)$-tuple
$$
(a_1,\ldots,a_{i-2},\,a_{i-1}-1,\,a_{i+1}-1,\,a_{i+2},\ldots,a_n).
$$
Clearly, the solution $(a_1,\ldots,a_n)$ can be obtained from this
$(n-1)$-tuple by an operation~(\ref{FirstSur}).
Equation~(\ref{Casea}) implies that the matrix
$M_{n-1}(a_1,\ldots,a_{i-2},\,a_{i-1}-1,\,a_{i+1}-1,\,a_{i+2},\ldots,a_n)$
remains equal to $M_n(a_1,\ldots,a_n)$.

In the case (b), take the $(n-3)$-tuple
$$
(a_1,\ldots,a_{i-2},\,a_{i-1}+a_{i+2}-1,\,a_{i+3},\ldots,a_n).
$$
The solution $(a_1,\ldots,a_n)$ is then a result of the operation~(\ref{SecSur})
applied to the coefficient $a_{i-1}+a_{i+2}-1$.
Equation~(\ref{Caseb}) implies that 
$M_{n-3}(a_1,\ldots,a_{i-2},\,a_{i-1}+a_{i+2}-1,\,a_{i+3},\ldots,a_n)
=M_n(a_1,\ldots,a_n)$.

The above inverse operations~(\ref{FirstSur}) and~(\ref{SecSur})
can always be applied, unless
$n=2$, or unless $n=3$ and there are at least two consecutive $1$'s.

Theorem~\ref{TheMainThm} is proved.
\hfill{$\Box$}

\subsection{Simple corollaries}

An immediate consequence of Theorem~\ref{TheMainThm} is the following
upper bound for the coefficients.

\begin{cor}
\label{SimpleCor}
If $(a_1,a_2,\ldots,a_n)$ is a solution of one of Problems~I, II, or~III, then

(i) $a_i\leq{}n-5$ (Problem~I); 

(ii) $a_i\leq{}n-2$ (Problem~II);

(iii) $a_i\leq{}n$ (Problem~III).
\end{cor}

\begin{proof}
The operations~(\ref{SecSur}) cannot increase the values of the coefficients~$a_i$,
while the operations~(\ref{FirstSur}) simultaneously increase $n$ and two coefficients by $1$.
\end{proof}

The next corollary gives expressions for the total sum of the coefficients.

\begin{cor}
\label{SumCor}
(i) If $(a_1,a_2,\ldots,a_n)$ is a solution of one of Problems~I or~II
obtained from the initial solution $(a_1,a_2,a_3)=(1,1,1)$ by applying a sequence
of~$S$ operations~(\ref{FirstSur}) and $R$ operations~(\ref{SecSur}), 
then
\begin{equation}
\label{SumPbII}
\begin{array}{rcl}
a_1+a_2+\cdots+a_n
&=&3S+3R+3\\[2pt]
&=&
3n-6R-6.
\end{array}
\end{equation}

(ii) If $(a_1,a_2,\ldots,a_n)$ is a solution of Problem~III
obtained from one of the initial solutions $(a_1,a_2)=(2,1)$ or $(1,2)$ by applying a sequence
of~$S$ operations~(\ref{FirstSur}) and $R$ operations~(\ref{SecSur}), 
then
\begin{equation}
\label{SumPbI}
\begin{array}{rcl}
a_1+a_2+\cdots+a_n
&=&3S+3R+3\\[2pt]
&=&
3n-6R-3.
\end{array}
\end{equation}
\end{cor}

\begin{proof}
Both the operations~(\ref{SecSur}) and~(\ref{FirstSur}) add $3$ to the total sum
of the coefficients.
Furthermore, the operations~(\ref{FirstSur})
(resp.~(\ref{SecSur})) increase~$n$ by~$1$ (resp. by~$3$).
\end{proof}

Note that the numbers~$S$ and~$R$ depend only on the solution $(a_1,a_2,\ldots,a_n)$
(and independent of the choice of the sequence of operations producing the solution).
The simplest case $R=0$ is precisely that of totally positive solutions of Conway and Coxeter;
see Appendix.

\subsection{Solutions of Problem~I for small~$n$}\label{Pb1SnSec}

Let us give several examples constructed using the inductive procedure
provided by Theorem~\ref{TheMainThm}.
We start with the list of solutions of Problem~I for~$n\leq8$.

(a)
Part (i) of Corollary~\ref{SimpleCor} implies that
Problem~I has no solutions for $n\leq5$.

(b)
For $n=6$,
Problem~I has the unique solution 
$$
(a_1,a_2,a_3,a_4,a_5,a_6)=(1,1,1,1,1,1)
$$
obtained from~(\ref{ElSolTwo}) by one operation~(\ref{SecSur}).

(c)
For $n=7$, one has $7$ different solutions:
\begin{equation}
\label{PbIn7}
(a_1,a_2,a_3,a_4,a_5,a_6,a_7)=(2,1,2,1,1,1,1)
\end{equation}
and its cyclic permutations.

(d)
For $n=8$, one has $34$ different solutions, namely
\begin{equation}
\label{PbIn8}
(a_1,\ldots,a_8)=
(3,1,1,1,1,2,2,1),
\quad
(3,1,2,1,1,1,2,1),
\quad
(2,2,1,2,1,1,2,1),
\quad
(2,1,2,1,2,1,2,1),
\end{equation}
and their reflections and cyclic permutations.

\subsection{Solutions of Problem~II for small~$n$}\label{Pb2SnSec}

Below is the list of solutions of Problem~II for~$n\leq10$.

(a)
For $n\leq8$ all solutions of Problem~II are given by Conway-Coxeter's solutions,
and correspond to triangulations of $n$-gons; see Appendix.
The number of solutions for a given $n$ is thus equal to the Catalan number~$C_{n-2}$, where
$C_n=\frac{1}{n+1}\binom{2n}{n}$.

(b)
For $n=9$, in addition to $429$ Conway-Coxeter solutions, there is exactly one extra solution:
\begin{equation}
\label{PbIIn9}
(a_1,\ldots,a_9)=(1,1,1,1,1,1,1,1,1).
\end{equation}

(c)
For $n=10$, in addition to $1430$ Conway-Coxeter solutions, there are $15$ solutions:
\begin{equation}
\label{PbIIn10}
(a_1,\ldots,a_{10})=
(2,1,1,1,1,2,1,1,1,1),
\quad
(2,1,2,1,1,1,1,1,1,1),
\end{equation}
and their cyclic permutations.

In Section~\ref{SecTEx} we will give the dissections of $n$-gons corresponding to the above
examples.

\section{The combinatorial model: $3d$-dissections}\label{CombSec}

In this section we prove Theorem~\ref{SecondMainThm} deducing it
from Theorem~\ref{TheMainThm}.
Using the combinatorics of $3d$-dissections,
we then obtain the formulas for the numbers of surgery operations for a given solution
of Problems~I or~II.
Finally, we revisit the examples from Sections~\ref{Pb1SnSec} and~\ref{Pb2SnSec} 
and give their combinatorial realizations.

\subsection{Proof of Theorem~\ref{SecondMainThm}}

Part (i).
Consider a solution $(a_1,\ldots,a_n)$ of Problem~I or~II.
We want to prove that there exists a $3d$-dissection of an $n$-gon
such that its quiddity is precisely the chosen solution.

We proceed by induction on $n$.
By Theorem~\ref{TheMainThm}, the chosen solution can be 
obtained from the initial solution~(\ref{ElSolTwo})
by a series of operations~(\ref{FirstSur}) and~(\ref{SecSur}).
Consider the solution (of length $n-1$ or $n-3$) obtained by the same sequence
but without the last operation.
By induction assumption, this solution corresponds to some $3d$-dissection, say~$D$,
(of an $(n-1)$-gon or an $(n-3)$-gon).
There are then two possibilities.

(a)
If the last operation in the series is that of type~(\ref{FirstSur}), then the solution corresponds to
the angulation~$D$ with extra exterior triangle glued to the edge $(i,i+1)$.

(b) Suppose that the last operation is that of type~(\ref{SecSur}).
Consider the new $3d$-dissection obtained from~$D$ by the
following local surgery at vertex~$i$, along a chosen sub-polygon:
$$
 \xymatrix @!0 @R=0.32cm @C=0.45cm
 {
&&&\ar@{-}[dddddddd]\ar@{-}[rrd]&
\\
&\ar@{-}[rrddddddd]&&&& \\
\\
\ar@{-}[rrrddddd]&&&&&&&\ar@{-}[llllddddd]\\
\\
\ar@{-}[rrrddd]&&&&&&&&\ar@{-}[lllllddd]\\
\\
&\ar@{-}[rrd]&&&&&&& \ar@{-}[llllld]\\
&&&i\ar@{-}[rruuuuuuu]&
}
\qquad\qquad
 \xymatrix @!0 @R=0.32cm @C=0.45cm
 {
&&&&\ar@{-}[ldddddddd]\ar@{-}[rrd]&
\\
&\ar@{-}[rrddddddd]&&&&& \\
\\
\ar@{-}[rrrddddd]&&&&&&&&\ar@{-}[lddddd]\\
\\
\ar@{-}[rrrddd]&&&&&&&&&\ar@{-}[llddd]\\
\\
&\ar@{-}[rrd]&&&&&&&& \ar@{-}[lld]\\
&&&i'\ar@{-}[rd]&&&&i''\ar@{-}[luuuuuuu]\\
&&&&1\ar@{-}[r]&1\ar@{-}[rru]
}
$$
that inserts two new vertices $1,1$ between two copies of the vertex $i$.
This leads to a  $3d$-dissection of an $(n+3)$-gon which is exactly as in the right-hand-side
of~(\ref{SecSur}).

Conversely, given a $3d$-dissection of an $n$-gon, we need to show that
its quiddity is a solution of Problem~I or~II.
This follows from the obvious fact that any $3d$-dissection of an $n$-gon by
pairwise non-crossing diagonals has an {\it exterior} sub-polygon.
By ``exterior'' we mean a sub-polygon without diagonals which is
glued to the rest of the $3d$-dissection along one edge
$$
 \xymatrix @!0 @R=0.32cm @C=0.45cm
 {
&&1\ar@{-}[llddd]\ar@{-}[rr]
&&a_j \ar@{-}[rr]&&\\
&&&&&&\ar@{-}[ldddd]\\
&&&&&&&\ar@{-}[llddd]\\
1\ar@{-}[rdd]\\
&&&&&&& \ar@{-}[lld]\\
&1\ar@{-}[rd]&&&&a_i\ar@{=}[luuuuu]\\
&&1\ar@{-}[r]&1\ar@{-}[rru]
}
$$
Such a $3d$-dissection can be reduced by applying the inverse of one of 
the operations~(\ref{FirstSur}) or~(\ref{SecSur}).
We then proceed by induction.
 
Part (i) of the theorem follows, the proof of Part (ii) is similar.

\medskip

Theorem~\ref{SecondMainThm} is proved. $\Box$

\subsection{Counting the surgery operations}
Consider a solution of Problem~I or~II corresponding to some $3d$-dissection.
Denote by $N_d$ is the number of $3d$-gons in the $3d$-dissection.

\begin{prop}
\label{ContOpProp}
Given a solution of Problem~I or II constructed from~(\ref{ElSolTwo})
by a sequence of~$S$ operations~(\ref{FirstSur}) and~$R$ operations~(\ref{SecSur}),

(i)
The number $S$ counts the total number of sub-polygons except for the initial one:
\begin{equation}
\label{SfromD}
S=\sum_{d\leq\left[\frac{n}{3}\right]}N_d-1.
\end{equation}

(ii)
The number of operations of the second type is the weighted sum
\begin{equation}
\label{RfromD}
R=\sum_{d\leq\left[\frac{n}{3}\right]}
\left(d-1\right)
N_d.
\end{equation}
\end{prop}

In other words, to calculate $R$, one ignores the triangles, counts hexagons,
counts nonagons $2$ times, dodecagons~$3$ times, etc.

\begin{proof}
An operation~(\ref{FirstSur}) consists in a gluing a triangle.
It increases the total number of sub-polygons by $1$.
This implies~(\ref{SfromD}).

We have proved (see the proof of Theorem~\ref{SecondMainThm}) 
that an operation~(\ref{SecSur}) does not change the total number of
sub-polygons of a $3d$-dissection, but adds $3$ new vertices to one of the
existing sub-polygons.
Hence~(\ref{RfromD}).
\end{proof}

\subsection{Non-uniqueness}\label{NUSec}
Unlike triangulations, a quiddity does not determine the
corresponding $3d$-dissection.
Different $3d$-dissections
may correspond to the same quiddity.\footnote{
This remark and examples were communicated
to me by Alexey Klimenko.}

For instance, this is the case for the following $3d$-dissections of the octagon
$$
\xymatrix @!0 @R=0.30cm @C=0.5cm
 {
&1\ar@{-}[ldd]\ar@{-}[rr]&& 2\ar@{-}[rdd]\\
\\
2\ar@{-}[dd]\ar@{-}[rdddd]&&&&1\ar@{-}[dd]\\
\\
1\ar@{-}[rdd]&&&&2\ar@{-}[ldd]\ar@{-}[luuuu]\\
\\
&2\ar@{-}[rr]&& 1
}
\qquad\qquad
\xymatrix @!0 @R=0.30cm @C=0.5cm
 {
&1\ar@{-}[ldd]\ar@{-}[rr]&& 2\ar@{-}[rdd]\\
\\
2\ar@{-}[dd]\ar@{-}[rrruu]&&&&1\ar@{-}[dd]\\
\\
1\ar@{-}[rdd]&&&&2\ar@{-}[ldd]\ar@{-}[llldd]\\
\\
&2\ar@{-}[rr]&& 1
}
$$
Therefore, one cannot expect a one-to-one correspondence
between solutions of Problems~I--III and $3d$-dissections. 
This discrepancy becomes more flagrant when we consider
$3d$-dissections of $2n$-gons.
Indeed, the following $3d$-dissection of the  tetradecagon
(``Klimenko's dissection'')
$$
\xymatrix @!0 @R=0.30cm @C=0.33cm
 {
&&&&&&3\ar@{-}[dddddddddddd]\ar@{-}[llld]\ar@{-}[rrrd]&
\\
&&&2\ar@{-}[lldd]\ar@{-}[rrrddddddddddd]&&&&&& 1\ar@{-}[rrdd]\\
\\
&1\ar@{-}[ldd]&&&&&&&&&&2\ar@{-}[rdd]\ar@{-}[llllluuu]\\
\\
2\ar@{-}[dd]\ar@{-}[rdddd]&&&&&&&&&&&&1\ar@{-}[dd]\\
\\
1\ar@{-}[rdd]&&&&&&&&&&&&2\ar@{-}[ldd]\ar@{-}[llldddd]\\
\\
&2\ar@{-}[rrdd]&&&&&&&&&&1\ar@{-}[lldd]\\
\\
&&&1\ar@{-}[rrrd]&&&&&& 2\ar@{-}[llld]\\
&&&&&&3&
}
$$
is not centrally symmetric, but the corresponding quiddity is $7$-periodic.
In fact, it coincides with that of Example~\ref{FirstEx}.

\subsection{Examples for small~$n$}\label{SecTEx}
Let us give combinatorial entities of the examples from Section~\ref{AlgoSec}.

(a)
Consider again the solutions of Problem~I  for small~$n$; see Section~\ref{Pb1SnSec}.
For $n=6$ the unique solution $(a_1,\ldots,a_6)=(1,\ldots,1)$ is given by the hexagon
without interior diagonals.

For $n=7$ the unique modulo cyclic permutations solution~(\ref{PbIn7})
corresponds to a triangle glued to an hexagon
$$
\xymatrix @!0 @R=0.35cm @C=0.35cm
{
&&&1\ar@{-}[rrd]\ar@{-}[lld]
\\
&2\ar@{-}[ldd]\ar@{-}[rrrr]&&&& 2\ar@{-}[rdd]&\\
\\
1\ar@{-}[rdd]&&&&&& 1\ar@{-}[ldd]\\
\\
&1\ar@{-}[rrrr]&&&&1
}
$$

For $n=8$ the solutions of Problem~I correspond to dissections of the octagon
into hexagon and two triangles.
There are exactly $4$ such dissections (modulo reflections and rotations):
$$
\xymatrix @!0 @R=0.30cm @C=0.5cm
 {
&1\ar@{-}[ldd]\ar@{-}[rr]&& 2\ar@{-}[rdd]\\
\\
3\ar@{-}[dd]\ar@{-}[rrrr]\ar@{-}[rrruu]&&&&2\ar@{-}[dd]\\
\\
1\ar@{-}[rdd]&&&&1\ar@{-}[ldd]\\
\\
&1\ar@{-}[rr]&& 1
}
\quad
\xymatrix @!0 @R=0.30cm @C=0.5cm
 {
&1\ar@{-}[ldd]\ar@{-}[rr]&& 2\ar@{-}[rdd]\\
\\
3\ar@{-}[dd]\ar@{-}[rdddd]\ar@{-}[rrruu]&&&&1\ar@{-}[dd]\\
\\
1\ar@{-}[rdd]&&&&1\ar@{-}[ldd]\\
\\
&2\ar@{-}[rr]&& 1
}
\quad
\xymatrix @!0 @R=0.30cm @C=0.5cm
 {
&1\ar@{-}[ldd]\ar@{-}[rr]&& 2\ar@{-}[rdd]\\
\\
2\ar@{-}[dd]\ar@{-}[rrruu]&&&&1\ar@{-}[dd]\\
\\
2\ar@{-}[rrrdd]\ar@{-}[rdd]&&&&1\ar@{-}[ldd]\\
\\
&1\ar@{-}[rr]&& 2
}
\quad
\xymatrix @!0 @R=0.30cm @C=0.5cm
 {
&1\ar@{-}[ldd]\ar@{-}[rr]&& 2\ar@{-}[rdd]\\
\\
2\ar@{-}[dd]\ar@{-}[rrruu]&&&&1\ar@{-}[dd]\\
\\
1\ar@{-}[rdd]&&&&2\ar@{-}[ldd]\ar@{-}[llldd]\\
\\
&2\ar@{-}[rr]&& 1
}
$$
in full accordance with~(\ref{PbIn8}).

(b)
Consider now the solutions of Problem~II discussed in Section~\ref{Pb2SnSec}.
For $n=9$ the solution~(\ref{PbIIn9}) obviously corresponds to the nonagon
with no dissection.
For $n=10$ there are two possibilities:  two glued hexagons
and a triangle glued to a nonagon
\begin{equation}
\label{ExEqD}
\xymatrix @!0 @R=0.32cm @C=0.45cm
 {
&&&2\ar@{-}[dddddddd]\ar@{-}[lld]\ar@{-}[rrd]&
\\
&1\ar@{-}[ldd]&&&& 1\ar@{-}[rdd]\\
\\
1\ar@{-}[dd]&&&&&&1\ar@{-}[dd]\\
\\
1\ar@{-}[rdd]&&&&&&1\ar@{-}[ldd]\\
\\
&1\ar@{-}[rrd]&&&& 1\ar@{-}[lld]\\
&&&2&
}
\qquad\qquad
\xymatrix @!0 @R=0.32cm @C=0.45cm
 {
&&&2\ar@{-}[lllddd]\ar@{-}[lld]\ar@{-}[rrd]&
\\
&1\ar@{-}[ldd]&&&& 1\ar@{-}[rdd]\\
\\
2\ar@{-}[dd]&&&&&&1\ar@{-}[dd]\\
\\
1\ar@{-}[rdd]&&&&&&1\ar@{-}[ldd]\\
\\
&1\ar@{-}[rrd]&&&& 1\ar@{-}[lld]\\
&&&1&
}
\end{equation}
This corresponds (modulo cyclic permutations) to the solutions~(\ref{PbIIn10}).
The first of the above dissections,
i.e., the ``hexagonal''  one, will 
play an important role in Section~\ref{GeneratSec}.

\section{Problem~III and zero-trace equation}\label{DiopSec}

An elementary observation shows that
Problem~III is equivalent to a single Diophantine equation,
namely $\tr{}M_n(a_1,\ldots,a_n)=0$.
We show that $3d$-dissections allow one to construct all integer
zero-trace unimodular matrices.

\subsection{The ``Rotundus'' polynomial}
The trace of the matrix (\ref{SLEq})
is a beautiful cyclically invariant polynomial
in $a_1,\ldots,a_n$, that we denote by $R_n(a_1,\ldots,a_n)$.
The first examples are:
$$
\begin{array}{rcl}
R_1(a) &=& a,\\[2pt]
R_2(a_1,a_2) &=& a_1a_2-2,\\[2pt]
R_3(a_1,a_2,a_3) &=& a_1a_2a_3-a_1-a_2-a_3,\\[2pt]
R_4(a_1,a_2,a_3,a_4) &=& a_1a_2a_3a_4-a_1a_2-a_2a_3-a_3a_4-a_1a_4+2,\\[2pt]
R_5(a_1,a_2,a_3,a_4,a_5) &=& 
a_1a_2a_3a_4a_5\\
&&-a_1a_2a_3-a_2a_3a_4-a_3a_4a_5-a_1a_4a_5-a_1a_2a_5\\
&&+a_1+a_2+a_3+a_4+a_5.
\\
\end{array}
$$
The polynomial $R_n(a_1,\ldots,a_n)$ was called the ``Rotundus'' in~\cite{CoOv},
where it is proved that $R_n(a_1,\ldots,a_n)$ 
can also be calculated as the Pfaffian of a certain skew-symmetric matrix.
Note that $R_n(a_1,\ldots,a_n)$ is the polynomial part of the rational function
$$
a_1a_2\cdots{}a_n
\left(1-\frac{1}{a_1a_2}\right)
\left(1-\frac{1}{a_2a_3}\right)
\cdots
\left(1-\frac{1}{a_na_1}\right).
$$

\subsection{The ``Rotundus equation''}
An $n$-tuple of positive integers $(a_1,\ldots,a_n)$
is a solution of Problem~III if and only if $\tr{}M_n(a_1,\ldots,a_n)=0.$
In other words, we have the following.

\begin{prop}
Every solution of Problem~III is a solution of the equation
\begin{equation}
\label{RotEq}
R_n(a_1,\ldots,a_n)=0,
\end{equation}
and vice-versa.
\end{prop}

\begin{proof}
A trace zero element of $\SL(2,\Z)$ has eigenvalues $i$ and $-i$.
This is equivalent to the fact that it squares to $-\id$.
\end{proof}

\begin{rem}
Note also that every solution of Problem I or II satisfies the equation
$R_n(a_1,\ldots,a_n)=2$ or $-2$, respectively.
However, the converse is false: a solution of one of these equations
is not necessarily a solution of Problem I or II.
It is also easy to see that,
unlike~(\ref{RotEq}), the equation $R_n(a_1,\ldots,a_n)=\pm2$ has infinitely many
positive integer solutions for sufficiently large~$n$.
For instance, one has $R_n(a,1,1)=-2$ for any~$a$.
\end{rem}

\subsection{The list of solutions of Problem~III for small~$n$}\label{ExIIISn}

Let us give a complete list of solutions of Problem~III for~$n\leq6$.

(a)
For $n=2,3$, and $4$, all the solutions are given by centrally symmetric
triangulations of a quadrilateral (2), hexagon (6), and octagon (20), respectively.

(b)
For $n=5$, besides $70$ solutions
corresponding to centrally symmetric triangulations of the decagon 
(see Example~\ref{DecaEx} below),
one obtains $5$ additional solutions:
$$
(a_1,a_2,a_3,a_4,a_5)=(1,1,2,1,1)
$$
and its cyclic permutations.
The corresponding centrally symmetric dissection of a decagon is 
the hexagonal dissection in~(\ref{ExEqD}).

(c)
For $n=6$, besides $252$ solutions
corresponding to centrally symmetric triangulations of the dodecagon, one gets~$26$
additional solutions:
$$
(a_1,a_2,a_3,a_4,a_5,a_6)=
(3,1,2,1,1,1),\quad
(2,2,1,2,1,1),\quad
(2,1,2,1,2,1),
$$
their cyclic permutations and reflections.

We mention that the sequence $2,6,20,75,278,\ldots$ corresponding to the total number 
of solutions of Problem~III is not in the OEIS.

\section{The rotation index}\label{RotSec}

In this section we explain how to associate an $n$-gon in the projective line,
i.e., an element of the moduli space $\cM_{0,n}$,
to every solution of Problem~I or~II.

We then apply the Sturm theory of linear difference equations
to define a geometric invariant
of solutions of Problem~I, II, and~III.
It is given by the index of a star-shaped broken line in~$\R^2$, that
can also be  understood as the homotopy class of an 
$n$-gon in the projective line,
or  as the rotation number of the equation~(\ref{DEqEq}).
The defined invariant is a {\it (half)integer}.
We prove that the index actually counts the number of operations
of the second type~(\ref{SecSur})
needed for a solution to be obtained from the initial one.

\subsection{Index of a star-shaped broken line}
Recall the following geometric notions.

a) The index of a smooth closed plane curve is the number of 
rotations of its tangent vector.

b) A smooth oriented (parametrized) closed curve $\g(t)$ in~$\R^2$,
where $t\in[0,1]$ and $\g(t+1)=\g(t)$ is  {\it star-shaped} 
if it does not contain the origin, and the tangent vector $\dot{\g}(t)$ 
is transversal to  $\g(t)$, for all $t$.

c)
The index of a star-shaped curve can be calculated as the homotopy class
of the projection of $\g(t)$ to~$\RP^1$ in the tautological line bundle
$\R^2\setminus\{0\}\to\RP^1$,
i.e., the rotation number around the origin.

Definitions a)--c) obviously extend to piecewise smooth
curves, in particular to {\it broken lines}.

\begin{ex}
The index of the following star-shaped broken lines:
$$
\xymatrix @!0 @R=1cm @C=1cm
{
&\ar@{<-}[rd]\ar@{->}[ld]
\\
&0&&\\
&\ar@{->}[ru]\ar@{<-}[lu]
}
\qquad
\xymatrix @!0 @R=0.5cm @C=0.7cm
{
&&\ar@{<-}[rrrdd]\ar@{->}[lddd]
\\
&&&\ar@{->}[ll]&
\\
&&0&&&\\
&&&\ar@{<-}[ll]\ar@{->}[uu]&
\\
&&\ar@{->}[rrruu]\ar@{<-}[luuu]
}
$$
is equal to $1$ and $2$, respectively.
\end{ex}

Furthermore, if the curve is {\it antiperiodic}, that is if $\g(t+1)=-\g(t)$,
the index is still well defined, but takes half-integer values.

\begin{ex}
The index of the following star-shaped antiperiodic broken lines:
$$
\xymatrix @!0 @R=1cm @C=1cm
{
&\ar@{<-}[rd]\ar@{->}[ld]
\\
&0&&}
\qquad
\xymatrix @!0 @R=0.7cm @C=0.7cm
{
&\ar@{->}[dd]&&\ar@{->}[llld]&
\\
&&0&&\ar@{->}[lllu]\\
&&&\ar@{<-}[ll]\ar@{->}[uu]&
}
$$
is equal to $\frac{1}{2}$ and $\frac{3}{2}$, respectively.
\end{ex}

\subsection{The broken line of a matrix $M_n(a_1,\ldots,a_n)$}
Given a solution $(a_1,\ldots,a_n)$ of Problem~I, II, or~III,
let us construct a star-shaped broken line in~$\R^2$.
Consider the corresponding discrete Sturm-Liouville equation
$$
V_{i+1}=a_iV_i-V_{i-1},
$$
where the set of coefficients $a_i$ is understood as
an infinite $n$-periodic sequence~$(a_i)_{i\in\Z}$.
Choose  two linearly independent solutions,
$V^{(1)}=(V^{(1)}_i)_{i\in\Z}$ and $V^{(2)}=(V^{(2)}_i)_{i\in\Z}$.
One then has a sequence of points  in~$\R^2$:
$$
V_i=
\left(V^{(1)}_{i},
V^{(2)}_{i}\right),
$$
These points form a broken star-shaped line.
Indeed, the determinant
$$
W(V^{(1)},V^{(2)}):=
\left|
\begin{array}{cc}
V^{(1)}_{i+1}&V^{(1)}_{i}\\[4pt]
V^{(2)}_{i+1}&V^{(2)}_{i}
\end{array}
\right|,
$$
usually called the {\it Wronski determinant},
is constant, i.e., does not depend on $i$.
Therefore, the sequence of points $(V_i)_{i\in\Z}$ in~$\R^2$ always
rotates around the origin in the same (positive or negative, depending
on the choice of the two solutions) direction.
Note that a different choice of the solutions
$V^{(1)}$ and $V^{(2)}$ gives the same broken line,
modulo a linear coordinate transformation in~$\R^2$.

If $M_n(a_1,\ldots,a_n)=\Id$ (resp. $-\Id$), then
the broken line thus constructed is periodic, i.e., $V_{i+n}=V_{i}$ 
(resp. anti-periodic, $V_{i+n}=-V_{i}$).
We will be interested in
the {\it index} of this broken line.

\begin{rem}
Note that the index of an antiperiodic star-shaped broken line
is a well-defined half-integer.
If $M_n(a_1,\ldots,a_n)^2=-\Id$, then, using the doubling procedure,
we can still define the index of the corresponding star-shaped broken line
as a multiple of $\frac{1}{2}$.
\end{rem}

\begin{ex}
(a)
Consider the sequence $(a_1,\ldots,a_6)=(1,1,1,1,1,1)$, which is the solution of Problem~I
obtained from $(1,1,1)$ by applying one operation~(\ref{SecSur}).
Choosing the solutions with the initial conditions 
$(V^{(1)}_0,V^{(1)}_1)=(1,0)$ and $(V^{(2)}_0,V^{(2)}_1)=(0,1)$, one obtains the
following hexagon in~$\R^2$:
$\{(1,0),(0,1),(-1,1),(-1,0),(0,-1),(1,-1)\}$.
$$
\xymatrix @!0 @R=1cm @C=1cm
{
{\textcircled{\raisebox{-.3pt} {2}}}\ar@{->}[d]
&{\textcircled{\raisebox{-.3pt} {1}}}\ar@{->}[l]
\\
{\textcircled{\raisebox{-.3pt} {3}}}\ar@{->}[rd]&
0&
 {\textcircled{\raisebox{-.3pt} {0}}}\ar@{->}[lu]&\\
&{\textcircled{\raisebox{-.3pt} {4}}}\ar@{->}[r]& {\textcircled{\raisebox{-.3pt} {5}}}\ar@{->}[u]
}
$$
The index of this hexagon is $1$.

(b)
Consider the solution of Problem~II $(a_1,a_2,a_3,a_4)=(2,1,2,1)$
obtained from $(1,1,1)$ by applying one operation~(\ref{FirstSur}).
Choosing the solutions with the same initial conditions as above, one obtains the
following antiperiodic quadrilateral in~$\R^2$:
$\{(1,0),(0,1),(-1,2),(-1,1)\}$.

$$
\xymatrix @!0 @R=1cm @C=1cm
{
{\textcircled{\raisebox{-.3pt} {2}}}\ar@{->}[d]&\\
{\textcircled{\raisebox{-.3pt} {3}}}\ar@{->}[d]
&{\textcircled{\raisebox{-.3pt} {1}}}\ar@{->}[lu]
\\
{\textcircled{\raisebox{-.3pt} {4}}}&
0&
 {\textcircled{\raisebox{-.3pt} {0}}}\ar@{->}[lu]&
 }
$$
whose index is $\frac{1}{2}$.
\end{ex}

\subsection{The index of a solution}

\begin{prop}
\label{BrPr}
For a solution of Problem~I or~II
obtained from~(\ref{ElSolTwo}) by a sequence of $S$ operations~(\ref{FirstSur})
and $R$ operations~(\ref{SecSur}), the index of the corresponding broken
line is equal to~$\frac{1}{2}(R+1)$.
\end{prop}

\begin{proof}
We need to show that the operations of the first type applied to
solution of Problems~I and~II do not change the index of the corresponding
broken line, while the operations of the second type increase this index by $\frac{1}{2}$.

An operation~(\ref{FirstSur}) adds one additional point,
$V_i+V_{i+1}$, between the points 
$V_i$ and $V_{i+1}$ in the sequence of points $(V_i)_{i\in\Z}$.
The resulting sequence is $(\ldots,V_i,V_i+V_{i+1},V_{i+1},\ldots)$,
which has the same index as the initial one.

An easy computation shows that the operation~(\ref{SecSur}) 
transforms the sequence of points $(V_i)_{i\in\Z}$ as follows:
$$
(\ldots,V_{i-1},V_i,V_{i+1},\ldots)
\mapsto
(\ldots,V_{i-1},V_i,\,
a'_iV_i-V_{i-1},\,
(a'_i-1)V_i-V_{i-1},\,
-V_i,-V_{i+1},\ldots).
$$
Indeed, the sequence on the right-hand-side is a solution of
the equation~(\ref{DEqEq})
with coefficients 
$$
(a_1,\ldots,a'_i,\,1,\,\,1,\,a''_i,\ldots,a_n).
$$
Therefore, the operation~(\ref{SecSur}) rotates the picture by $180^\circ$
and thus increases the index by~$\frac{1}{2}$.
\end{proof}

\subsection{Non-osculating solutions and triangulations}

Similarly to the classical Sturm theory of linear differential and difference equations,
it is natural to introduce the following notion.

\begin{defn}
A solution of Problem~II whose the index is equal to~$\frac{1}{2}$,
 is called {\it non-osculating}.
\end{defn}

In other words, a solution of Problem~II is non-osculating
the number~$R$ of surgery operations~(\ref{SecSur})
needed to obtain this solution from the elementary solution
$(a_1,a_2,a_3)=(1,1,1)$ equals zero.
Note that
solutions of Problem~I cannot be non-osculating because $R$ is odd in this case.

The class of non-osculating solutions of Problem~II is precisely the
totally positive solutions of
Conway and Coxeter (see Appendix below).
Indeed, if the number~$R$ equals zero, then the $3d$-dissection
is a triangulation, cf. Proposition~\ref{ContOpProp}.

Similarly, one can define the class of non-osculating solutions of Problem~III
as that corresponding to symmetric triangulations of a $2n$-gon.
Again, the non-osculating property is equivalent to that of total positivity.

\section{An application: oscillating tame friezes}\label{CoCoSec}

We briefly introduce the notion of tame ``oscillating'' Coxeter friezes.
We show that this notion is equivalent to solutions of Problem~II.
Theorems~\ref{TheMainThm} and~\ref{SecondMainThm} then provide a classification
of tame oscillating friezes.
It is easy to see that
oscillating Coxeter friezes satisfy the main properties of the classical friezes,
such as Coxeter's glide symmetry.

\subsection{Classical Coxeter friezes}

Coxeter's frieze~\cite{Cox} is an array of $(n-1)$ infinite rows of  {\it positive integers},
with the first and the last rows consisting of $1$'s.
Consecutive rows are shifted, and the so-called Coxeter {\it unimodular rule}:
$$
\begin{array}{ccc}
&b&\\
a&&d\\
&c&
\end{array},
\qquad\qquad
ad-bc=1,
$$
is satisfied for every elementary $2\times2$ ``diamond''.

The Conway-Coxeter theorem~\cite{CoCo} provides a classification
of  Coxeter's friezes.
Every frieze corresponds to a triangulated $n$-gon,
the rows $2$ and $n-2$ being the quiddity of a triangulation;
see Definition~\ref{QuiD}.

\begin{ex}
For example, the frieze
$$
 \begin{array}{cccccccccccccccc}
\cdots&&1&& 1&&1&&1&&1
 \\[2pt]
&1&&3&&1&&2&&2&&\cdots
 \\[2pt]
\cdots&&2&&2&&1&&3&&1
 \\[2pt]
&1&& 1&&1&&1&&1&&\cdots
\end{array}
$$
is the unique (up to a cyclic permutation) Coxeter frieze for $n=5$.
It corresponds to the quiddity $(a_1,a_2,a_3,a_4,a_5)=(1,3,1,2,2)$.
\end{ex}

We refer to~\cite{Sop} for a survey on friezes and their connection
to various topics.

\subsection{Tameness}

Let us relax the positivity assumption.
Then frieze patterns may become undetermined, as discussed in~\cite{CR},
or very ``wild'', and the classification
of such friezes is out of reach; cf.~\cite{Cun}.
An important property that we keep is that of tameness, first introduced in~\cite{BR}.

\begin{defn}
A frieze is {\it tame} if the determinant of every elementary $3\times3$-diamond vanishes.
\end{defn}

\begin{rem}
Note that every classical Coxeter frieze is tame.
This follows easily from the positivity assumption.
\end{rem}

\subsection{Friezes corresponding to solutions of Problems~II and~III}

It turns out that
solutions of Problems~II and~III precisely correspond to tame friezes with 
$(a_i)_{i\in\Z}$ in the $2$nd row.
More precisely, we have the following

\begin{prop}
There is a one-to-one correspondence between

(i) Solutions of Problem~II and tame friezes with the $2$nd row all positive integers;

(ii) Solutions of Problem~III and tame friezes with even~$n$ and the $2$nd row of positive integers
which are invariant under reflection in the middle row.
\end{prop}

\begin{proof}
The following fact was noticed in~\cite{CoCo} for classical
Coxeter friezes, and proved in~\cite{SVRS} for tame friezes.

\begin{lem}
Every diagonal of a tame frieze is a solution of the equation~(\ref{DEqEq})
with coefficients $(a_i)_{i\in\Z}$ in the $2$nd row of the frieze.
\end{lem}

Part (i) readily follows from the lemma, while Part~(ii) is then a consequence of
Coxeter's glide symmetry.
\end{proof}

\begin{ex}
The solution of Problem~III with $(a_1,a_2,a_3,a_4,a_5)=(1,1,2,1,1)$
generates the following tame frieze with $n=10$:
$$
 \begin{array}{ccccccccccccccccccccccccc}
\cdots&&1&& 1&&1&&1&&1&&1&&1&&1&&1&&1
 \\[2pt]
&2&&1&&1&&1&&1&&2&&1&&1&&1&&1&&\cdots
 \\[2pt]
\cdots&&1&&0&&0&&0&&1&&1&&0&&0&&0&&1
 \\[2pt]
&0&&\!-1&&\!-1&&\!-1&&\!-1&&0&&\!-1&&\!-1&&\!-1&&\!-1&&\cdots
 \\[2pt]
\cdots&&\!-1&&\!-2&&\!-1&&\!-2&&\!-1&&\!-1&&\!-2&&\!-1&&\!-2&&\!-1
 \\[2pt]
&0&&\!-1&&\!-1&&\!-1&&\!-1&&0&&\!-1&&\!-1&&\!-1&&\!-1&&\cdots
 \\[2pt]
\cdots&&1&&0&&0&&0&&1&&1&&0&&0&&0&&1
 \\[2pt]
&2&&1&&1&&1&&1&&2&&1&&1&&1&&1&&\cdots
 \\[2pt]
\cdots&&1&& 1&&1&&1&&1&&1&&1&&1&&1&&1
\end{array}
$$
Every row is $5$-periodic, and the frieze is symmetric under the
reflection.
\end{ex}

\begin{rem}
(a)
The condition of {\it total positivity} for a solution
$(a_1,\ldots,a_n)$ of one of the Problems~I-III
is precisely the condition that every entry of the frieze pattern
with quiddity $(a_1,\ldots,a_n)$ is positive.
This condition was introduced by Coxeter (see~\cite{Cox,CoCo}),
and it is usually assumed in the literature on friezes; see~\cite{Sop}.
We will discuss the condition of total positivity in more details
in Appendix.

(b)
A frieze pattern can be viewed as the ``matrix'' of a
Sturm-Liouville operator~(\ref{DEqEq}) acting on the infinite-dimensional space
of sequences of numbers.
This point of view relates friezes to many different areas of mathematics.
In particular, it allows one to apply the tools
of linear algebra; see~\cite{SVRS},
and is useful for the spectral theory of linear difference operators; see~\cite{Kri}.
\end{rem}

\section{Towards $3d$-dissections of elements of $\PSL(2,\Z)$}\label{GeneratSec}
In this section we work with the group 
$\PSL(2,\Z)=\SL(2,\Z)/\left\langle\pm\Id\right\rangle$,
called the modular group.
Our goal is to define the notions of quiddity
and $3d$-dissection associated with an element of $\PSL(2,\Z)$.
The main statement of this section is formulated as conjecture, 
we hope to develop the subject elsewhere.

The notions of quiddity
and of $3d$-dissection
of an element of $\PSL(2,\Z)$ deserve a further study,
and need to be better understood.
In particular, it would be interesting to understand their relations
with the Farey graph and the hyperbolic plane.
This could eventually provide a proof of the conjecture.

\subsection{The generators of $\PSL(2,\Z)$}
It is a classical fact that the group 
$\PSL(2,\Z)$ can be generated by two elements, say $S$ and $L$, satisfying
$$
S^2=1,
\qquad
L^3=1,
$$
and with no other relations.
More formally, $\PSL(2,\Z)$
is isomorphic to the free product of two cyclic groups
$\Z/2\Z*\Z/3\Z$.

A possible choice of the generators is given by the following two matrices
that, abusing the notation, will also be denoted by~$S$ and~$L$:
$$
S=\left(
\begin{array}{cc}
0&1\\[4pt]
-1&0
\end{array}
\right),
\qquad
L=
\left(
\begin{array}{cc}
1&-1\\[4pt]
1&0
\end{array}
\right).
$$
These are generators of $\SL(2,\Z)$, and of $\PSL(2,\Z)$, modulo the center.
The matrices~$S$ and~$L$ are a square root and a cubic root of $-\Id$, respectively.

Another choice of generators which is often used is
$S$ and $T:=LS$, so that
$$
T=\left(
\begin{array}{cc}
1&1\\[4pt]
0&1
\end{array}
\right)
$$
the ``transvection matrix''.
Note however, that $S$ and $T$ are not free generators.

\subsection{Reduced positive decomposition}
 Every element of $\SL(2,\Z)$
can be written in the form~(\ref{SLEq}), for some
$n$-tuple of positive integers  $(a_1,\ldots,a_n)$.
This follows from the simple observation
(already mentioned in the introduction) that $S=M_5(1,1,2,1,1)$,
see~(\ref{RelSL}),
while the second generator~$L$ of $\SL(2,\Z)$ is already in this form.

Furthermore,
for an element of $\PSL(2,\Z)$ one can choose a canonical, 
or {\it reduced} presentation in this form.

\begin{defn}
An $n$-tuple of positive integers $(a_1,\ldots,a_n)$ is called {\it reduced} if it does not contain
subsequences $a_i,\,1\,,a_{i+2}$ with $a_i,a_{i+2}>1$, and $a_i,\,1,\,1\,,a_{i+3}$
with arbitrary $a_i,a_{i+3}$.
\end{defn}

Every $n$-tuple can be brought into reduced form by a sequence of operations inverse 
to the surgery operations~(\ref{FirstSur}) and~(\ref{SecSur}).
The matrix $M_n(a_1,\ldots,a_n)$ can only change its sign under these operations.
A reduced $n$-tuple can only have one or two $1$'s in the beginning or in the end.

We omit here a straightforward but tedious proof of the following
{\it uniqueness statement}: for every element $A\in\SL(2,\Z)$
there exists a unique reduced $n$-tuple of positive integers $(a_1,\ldots,a_n)$
such that $A=M_n(a_1,\ldots,a_n)$.
Roughly speaking, this uniqueness means that the operations~(\ref{FirstSur}) and~(\ref{SecSur})
commute.

\subsection{The quiddity and $3d$-dissection of an element $A\in\PSL(2,\Z)$}
Given an element $A\in\PSL(2,\Z)$,
we suggest the following construction.

Writing $A$ and $A^{-1}$ in the reduced form~(\ref{SLEq})
$$
A=M_k(a_1,\ldots,a_k),
\qquad
A^{-1}=M_\ell(a'_1,\ldots,a'_\ell),
$$
one obtains a $(k+\ell)$-tuple of positive integers
$(a_1,\ldots,a_k,\,a'_1,\ldots,a'_\ell)$, that we call the {\it quiddity of} $A$.

Furthermore,
taking into account the fact that 
$$
M_{k+\ell}(a_1,\ldots,a_k,\,a'_1,\ldots,a'_\ell)=
M_k(a_1,\ldots,a_k)\,M_\ell(a'_1,\ldots,a'_\ell)=\pm\Id,
$$
by Theorem~\ref{SecondMainThm},
this is a quiddity of some
$3d$-dissection.

\begin{conj}
\label{TheCon}
Every element $\PSL(2,\Z)$ corresponds to a unique $3d$-dissection.
\end{conj}

Recall that a quiddity does not necessarily determine a $3d$-dissection
(cf. Section~\ref{NUSec}).
The above conjecture means that this non-uniqueness phenomenon never
occurs for $3d$-dissections associated to elements of $\PSL(2,\Z)$.

A consequence of the above conjecture is that
very element $\PSL(2,\Z)$ has some
index, or ``rotation number'', see Section~\ref{RotSec}.

\subsection{Examples}
Let us give a few examples.

(a)
As follows from~(\ref{RelSL}),
the matrix~$S$ corresponds to
the quiddity of the hexagonal dissection of a decagon:
$$
\xymatrix @!0 @R=0.32cm @C=0.45cm
 {
&&&{2}\ar@{-}[dddddddd]\ar@{-}[lld]\ar@{-}[rrd]&
\\
&{1}\ar@{-}[ldd]&&&& {1}\ar@{-}[rdd]\\
\\
{1}\ar@{-}[dd]&&&&&&{1}\ar@{-}[dd]\\
\\
1\ar@{-}[rdd]&&&&&&1\ar@{-}[ldd]\\
\\
&1\ar@{-}[rrd]&&&& 1\ar@{-}[lld]\\
&&&2&
}
$$
The index is~$\frac{3}{2}$.

(b)
For the matrix $T$
one has $T=M_3(2,1,1)$ (up to a sign) and $T^{-1}=M_4(1,1,2,1)$.
This leads to the dissection of a heptagon:
$$
\xymatrix @!0 @R=0.35cm @C=0.35cm
{
&&&1\ar@{-}[rrd]\ar@{-}[lld]
\\
&2\ar@{-}[ldd]\ar@{-}[rrrr]&&&& 2\ar@{-}[rdd]&\\
\\
1\ar@{-}[rdd]&&&&&& 1\ar@{-}[ldd]\\
\\
&1\ar@{-}[rrrr]&&&&1
}
$$
The index is~$1$.

(c) 
Consider the following elements
$$
A=\left(
\begin{array}{cc}
2&1\\[4pt]
1&1
\end{array}
\right),
\qquad
B=\left(
\begin{array}{cc}
5&2\\[4pt]
2&1
\end{array}
\right)
$$ 
known as {\it Cohn matrices}.
These matrices play an important role in the theory of Markov numbers; see~\cite{Aig}.
One has the following presentations:
$$
A=M_4(2,2,1,1),
\quad
B=M_5(3,2,2,1,1),
\qquad
A^{-1}=M_4(1,1,3,1),
\quad
B^{-1}=M_5(1,1,4,2,1).
$$
The corresponding quiddities are those of the dissected octagon and decagon:
$$
\xymatrix @!0 @R=0.30cm @C=0.5cm
 {
&2\ar@{-}[ldd]\ar@{-}[rrrdd]\ar@{-}[rr]&& 1\ar@{-}[rdd]\\
\\
2\ar@{-}[dd]\ar@{-}[rrrr]&&&&3\ar@{-}[dd]\\
\\
1\ar@{-}[rdd]&&&&1\ar@{-}[ldd]\\
\\
&1\ar@{-}[rr]&& 1
}
\qquad
\qquad
 \xymatrix @!0 @R=0.32cm @C=0.45cm
 {
&&&4\ar@{-}[lld]\ar@{-}[dddddddd]\ar@{-}[rrd]&
\\
&2\ar@{-}[ldd]\ar@{-}[ldddd]&&&& 1\ar@{-}[rdd]\\
\\
1\ar@{-}[dd]&&&&&&1\ar@{-}[dd]\\
\\
3\ar@{-}[rrruuuuu]\ar@{-}[rdd]&&&&&&1\ar@{-}[ldd]\\
\\
&2\ar@{-}[rrd]\ar@{-}[rruuuuuuu]&&&& 1\ar@{-}[lld]\\
&&&2&
}
$$
The index of both elements, $A$ and $B$, is~$1$.

\section{Appendix: Conway-Coxeter quiddities and Farey sequences}\label{MotSec}

This section is an overview and does not contain new results.
We describe the Conway-Coxeter theorem,
formulated in terms of matrices $M_n(a_1,\ldots,a_n)$, and a similar result
in the case of Problem~III, obtained in~\cite{CoOv}.
We also briefly discuss the relation to Farey sequences.

In the seminal paper~\cite{CoCo}, Conway and Coxeter classified
solutions of Problem~II\footnote{Conway and Coxeter worked with
so-called frieze patterns (see Section~\ref{CoCoSec}), 
but the equivalence of their result to the classification of totally positive
solutions of Problem~II is a simple observation; see~\cite{BR,SVRS}.}
that satisfy a certain condition of
total positivity.
These are precisely the solutions obtained from the initial
solution $(a_1,a_2,a_3)=(1,1,1)$ by a sequence of operations~(\ref{FirstSur}).
Their classification of totally positive solutions beautifully relates Problem~II
to such classical subjects as triangulations of $n$-gons.
Furthermore, the close relation of the 
topic to Farey sequences was already mentioned in~\cite{Cox}.
It turns out that the Conway-Coxeter theorem implies some results of~\cite{HS}
about the index of a Farey sequence.

\subsection{Total positivity}

The class of totally positive solutions of Problem~II can be defined
in several equivalent ways.
Coxeter~\cite{Cox} (and Conway and Coxeter~\cite{CoCo}) 
assumed that all the entries of the corresponding frieze are positive. 

Another simple definition is based on the properties of solutions of the Sturm-Liouville equation.

\begin{defn}
\label{DefO}
A solution $(a_1,\ldots,a_n)$ of Problem~II is called totally positive if
there exists a solution $(V_i)_{i\in\Z}$ of the equation~(\ref{DEqEq})
that  does not change its sign on the interval $[1,\ldots,n]$, i.e.,
the sequence of~$n$ numbers $(V_1,V_2,\ldots,V_{n})$
is either positive, or negative.
\end{defn}

In the context of Sturm oscillation theory,
this case is often called ``non-oscillating'', or ``disconjugate''. 
The index of the corresponding broken line is equal to~$\frac{1}{2}$, see Section~\ref{RotSec}.

\begin{rem}
Note that since $M_n(a_1,\ldots,a_n)=-\Id$, every solution is $n$-anti-periodic,
so that it must change sign on the interval $[1,n+1]$.
\end{rem}

Let us give an equivalent combinatorial definition.
Consider the following tridiagonal $i\times{}i$-determinant
$$
K_i(a_1,\ldots,a_{i})=
\left|
\begin{array}{cccccc}
a_1&1&&&\\[4pt]
1&a_{2}&1&&\\[4pt]
&\ddots&\ddots&\!\!\ddots&\\[4pt]
&&1&a_{i-1}&\!\!\!\!\!1\\[4pt]
&&&\!\!\!\!\!1&\!\!\!\!a_{i}
\end{array}
\right|.
$$
This polynomial is nothing but the celebrated {\it continuant},
already known by Euler, and considered by many authors.
It was proved by Coxeter~\cite{Cox}
that the entries or a frieze pattern can be calculated as continuants
of the entries of the second row.

It is also well-known, see, e.g.,~\cite{BR}, (and can be easily checked directly) that the
entries of the $2\times2$ matrix~(\ref{SLEq}) can be 
explicitly calculated in terms of these determinants as follows:
$$
M_n(a_1,\ldots,a_n)=
\left(
\begin{array}{cc}
K_n(a_1,\ldots,a_{n})&-K_{n-1}(a_2,\ldots,a_{n})\\[6pt]
K_{n-1}(a_1,\ldots,a_{n-1})&-K_{n-2}(a_2,\ldots,a_{n-1})
\end{array}
\right).
$$
The condition $M_n(a_1,\ldots,a_n)=-\Id$
implies that 
$$
\begin{array}{rcl}
K_{n}(a_i,\ldots,a_{i+n-1})&=&-1,\\[4pt]
K_{n-1}(a_i,\ldots,a_{i+n-2})&=&0,\\[4pt]
K_{n-2}(a_i,\ldots,a_{i+n-3})&=&1,
\end{array}
$$
for all~$i$. 

The following definition is equivalent to Definition~\ref{DefO}.
\begin{defn}
A solution $(a_1,\ldots,a_n)$ of Problem~II is {\it totally positive} if
$$
K_{j+1}(a_i,\ldots,a_{i+j})>0
$$
for all $j\leq{}n-3$ and all $i$.
Note that we use the cyclic ordering of the $a_i$.
\end{defn}

\subsection{Triangulated $n$-gons}

The Conway-Coxeter result states that
totally positive solutions of Problem~II 
are in one-to-one correspondence
with triangulations of $n$-gons.

Given a triangulation of an $n$-gon,
let~$a_i$ be the number of triangles adjacent to the $i^\thup$ vertex.
This yields an $n$-tuple of positive integers, $(a_1,\ldots,a_n)$.
Conway and Coxeter called an $n$-tuple obtained from such a triangulation
 a quiddity.

\begin{thmc} (see~\cite{CoCo}).
Any quiddity of a triangulation is
a totally positive solution of Problem~II,
and every totally positive solution of Problem~II arises in this way.
\end{thmc}
 
A direct proof of the Conway-Coxeter theorem in terms of $2\times2$-matrices 
is given in~\cite{CH,BR}.
For a simple direct proof, see also~\cite{Hen}.

\begin{ex}
For $n=5$, the triangulation of the pentagon
$$
\xymatrix @!0 @R=0.50cm @C=0.5cm
{
&&3\ar@{-}[rrd]\ar@{-}[lld]\ar@{-}[lddd]\ar@{-}[rddd]&
\\
1\ar@{-}[rdd]&&&& 1\ar@{-}[ldd]\\
\\
&2\ar@{-}[rr]&& 2
}
$$
generates a solution $(a_1, a_2,a_3, a_4,a_5)=(1, 3, 1, 2, 2)$ of Problem~II.
All other solutions for $n=5$ are obtained by cyclic permutations of this one.
 \end{ex}

\subsection{Gluing triangles}
Obviously, every triangulation of an $n$-gon can be obtained from a triangle by
adding new exterior triangles.

\begin{ex}
Gluing a triangle to the above triangulated pentagon
$$
\xymatrix @!0 @R=0.50cm @C=0.5cm
{
&&3\ar@{-}[rrd]\ar@{-}[lld]\ar@{-}[lddd]\ar@{-}[rddd]&
\\
1\ar@{-}[rdd]&&&& 11\ar@<3pt>@{-}[ldd]\ar@{-}[ldd]\\
\\
&2\ar@{-}[rr]&& 21\ar@{-}[rr]&& 1\ar@{-}[luu]
}
\qquad
\simeq
\qquad
\xymatrix @!0 @R=0.50cm @C=0.5cm
{
&&3\ar@{-}[rrd]\ar@{-}[lld]\ar@{-}[llddd]\ar@{-}[dddd]&
\\
1\ar@{-}[dd]&&&& 2\ar@{-}[llddd]\ar@{-}[dd]\\
\\
2\ar@{-}[rrd]&&&& \,1,\ar@{-}[lld]\\
&&3
}
$$
one obtains the solution $(1, 3, 2,1, 3, 2)=(1, 3, 1+1,\,1,\, 2+1, 2)$ of Problem~II,
for $n=6$.
 \end{ex}

An operation~(\ref{FirstSur}) applied to a quiddity consists in
gluing a triangle to a triangulated $n$-gon, so that 
the Conway-Coxeter theorem implies the following statement (see also~\cite{CH}, Theorem 5.5).

\begin{cor}
Every totally positive solution of Problem~II can be obtained from
the initial solution $(a_1,a_2,a_3)=(1,1,1)$ by a sequence of operations~(\ref{FirstSur}).
Conversely, every sequence of operations~(\ref{FirstSur}) applied to this initial solution
is a totally positive solution of Problem~II.
\end{cor}

For a clear and detailed discussion; see~\cite{BR}.

\subsection{Indices of Farey sequences as Conway-Coxeter quiddities}

Relation to Farey sequences and negative continued fractions was already mentioned
by Coxeter~\cite{Cox} (see also~\cite{SVS}).

Rational numbers in $[0,1]$ whose denominator
does not exceed $N$ written in a form of irreducible fractions
form the {\it Farey sequence}
of order $N$.
Elements of the Farey sequence, $v_1=\frac{a_1}{b_1}$ and $v_2=\frac{a_2}{b_2}$,
are joined by an edge if and only if
$$
|a_1b_2-a_2b_1|=1.
$$
This leads to the classical notion of {\it Farey graph}.
The Farey graph is often embedded into the hyperbolic plane,
the edges being realized as geodesics joining rational points on the ideal boundary.

  \begin{figure}[!h]
\begin{center}
\includegraphics[width=8cm]{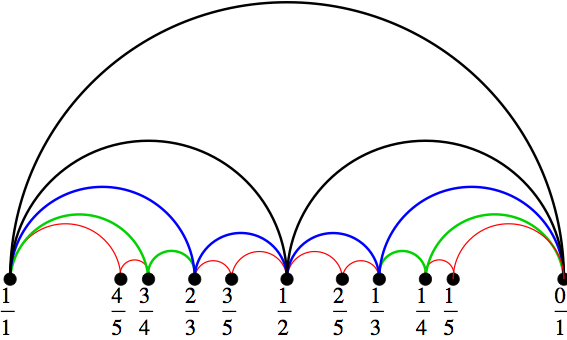}
\end{center}
\caption{The Farey sequence of order $5$ and the triangulated hendecagon.}
\label{SerFig}
\end{figure}

The main properties of Farey sequences can be found in~\cite{HaWr}.
A simple but important property is that
every Farey sequence forms a triangulated polygon in the Farey graph.
A Conway-Coxeter quiddity is then precisely the index of a Farey sequence,
defined in~\cite{HS}.

The Conway-Coxeter theorem implies the following.

\begin{cor}
A solution $(a_1,\ldots,a_n)$ of Problem~II is totally positive
if and only if
$$
a_1+a_2+\cdots+a_n
=
3n-6.
$$
\end{cor}

Indeed, the total number of triangles in a triangulation is~$n-2$,
and each triangle has three angles that contribute to a quiddity.

\begin{rem}
The above formula is equivalent to Theorem~1 of~\cite{HS}.
Moreover, it holds
not only for the complete Farey sequence, but also for 
an arbitrary {\it path in the Farey graph}.
Consider the Farey sequence of order $5$
presented in Figure~\ref{SerFig}. It has many different shorter paths,
for instance, $\left\{\frac{1}{1},\frac{2}{3},\frac{3}{5},\frac{1}{2},\frac{1}{3},\frac{0}{1}\right\}$.
\end{rem}

\subsection{Totally positive solutions of Problem~III}

A solution $(a_1,\ldots,a_n)$ of Problem~III is totally positive if its double
$(a_1,\ldots,a_n,a_1,\ldots,a_n)$ is a totally positive solution of Problem~II.
Every totally positive solution can be obtained from one of the solutions
$(a_1,a_2)=(1,2)$, or $(2,1)$ by a sequence of operations~(\ref{FirstSur}).

The Conway-Coxeter theorem implies that
there is a one-to-one correspondence between totally positive solutions of 
Problem~III and {\it centrally symmetric} triangulations of $2n$-gons;
see also~\cite{CoOv}.

\begin{ex}
\label{DecaEx}
There exist $70$ different
centrally symmetric triangulations of the decagon, for instance
$$
 \xymatrix @!0 @R=0.32cm @C=0.45cm
 {
&&&5\ar@{-}[lllddd]\ar@{-}[lld]\ar@{-}[dddddddd]\ar@{-}[rrd]&
\\
&1\ar@{-}[ldd]&&&& 2\ar@{-}[rdd]\\
\\
2\ar@{-}[dd]&&&&&&2\ar@{-}[dd]\ar@{-}[lllddddd]\\
\\
2\ar@{-}[rrruuuuu]\ar@{-}[rdd]&&&&&&2\ar@{-}[ldd]\ar@{-}[lllddd]\\
\\
&2\ar@{-}[rrd]\ar@{-}[rruuuuuuu]&&&& 1\ar@{-}[lld]\\
&&&5\ar@{-}[rruuuuuuu]&
}
\qquad
\xymatrix @!0 @R=0.32cm @C=0.45cm
 {
&&&4\ar@{-}[lllddd]\ar@{-}[lld]\ar@{-}[dddddddd]\ar@{-}[rrd]&
\\
&1\ar@{-}[ldd]&&&& 3\ar@{-}[rdd]\ar@{-}[rdddd]\\
\\
3\ar@{-}[dd]\ar@{-}[rdddd]&&&&&&1\ar@{-}[dd]\\
\\
1\ar@{-}[rdd]&&&&&&3\ar@{-}[ldd]\ar@{-}[lllddd]\\
\\
&3\ar@{-}[rrd]\ar@{-}[rruuuuuuu]&&&& 1\ar@{-}[lld]\\
&&&4\ar@{-}[rruuuuuuu]&
}
\qquad
\xymatrix @!0 @R=0.32cm @C=0.45cm
 {
&&&4\ar@{-}[lllddd]\ar@{-}[lld]\ar@{-}[dddddddd]\ar@{-}[rrd]&
\\
&1\ar@{-}[ldd]&&&& 2\ar@{-}[rdd]\ar@{-}[rdddd]\\
\\
4\ar@{-}[dd]\ar@{-}[rdddd]&&&&&&1\ar@{-}[dd]\\
\\
1\ar@{-}[rdd]&&&&&&4\ar@{-}[ldd]\ar@{-}[lllddd]\ar@{-}[llluuuuu]\\
\\
&2\ar@{-}[rrd]&&&& 1\ar@{-}[lld]\\
&&&4\ar@{-}[llluuuuu]&
}
$$
The corresponding sequences
$
(a_1,a_2,a_3,a_4,a_5)=
(5,2,2,2,1), \,  (4,3,1,3,1), \,  (4,2,1,4,1),\ldots
$
are totally positive solutions of Problem~III.
\end{ex}

The total number of totally positive solutions of Problem~III
is given by the central binomial coefficient $\binom{2n}{n}=1, 2, 6, 20, 70, 252, 924,\ldots$

\bigbreak \noindent
{\bf Acknowledgements}.
I am grateful to Charles Conley, Vladimir Fock,
Alexey Klimenko, Ian Marshall, Sophie Morier-Genoud and Sergei Tabachnikov
for multiple stimulating and enlightening discussions.
I am also grateful to the referees for a number of useful comments and suggestions.

\end{document}